\documentclass[11pt]{amsart}
\usepackage{graphicx}
\usepackage{amssymb}
\usepackage{cite}
\usepackage{hyperref}
\usepackage{enumerate}
\usepackage{txfonts} 
\usepackage{mathtools}
\usepackage[usenames,dvipsnames]{color}

\numberwithin{equation}{section} 

\newtheorem{theorem}{Theorem}[section]
\newtheorem{theorem*}{Theorem}
\newtheorem{lemma}[theorem]{Lemma}

\newtheorem{corollary}[theorem]{Corollary}

\newtheorem{remark}[theorem]{Remark}
\newtheorem{remark*}[theorem*]{Remark}
\newtheorem{definition}[theorem]{Definition}

\def\R{{\mathbb R}}

\def\cB{{\mathcal B}}

\def\1{\left(}
\def\2{\right)}
\def\3{\left\{}
\def\4{\right\}}
\def\8{\infty}

\def\ss{\subseteq}

\DeclareMathOperator*{\osc}{osc}

\begin{document}
\title{Comparison Principles for nonlocal Hamilton-Jacobi equations}

\author[G. D\'avila]{Gonzalo D\'avila}
\address{Universidad T\'ecnica Federico Santa Mar\'ia, Departamento de Matem\'atica, Avenida Espa\~na 1680, Valpara\'iso, Chile}
\email{gonzalo.davila@usm.cl}
\begin{abstract}
We prove the comparison principle for viscosity sub and super solutions of degenerate nonlocal operators with general nonlocal gradient nonlinearities. The proofs apply to purely Hamilton-Jacobi equations of order $0<s<1$.
\end{abstract}

\maketitle

\section{Introduction}\label{SecIntro}

Let us recall the classic first order Hamilton-Jacobi equation
\begin{align}\label{HJ}
\left \{ \begin{array}{rll}
u_t+H(x,u,Du) & = 0 \quad & \mbox{in} \ \R^N\times(0,T) \\
u(x,0) & = u_0(x) \quad & \mbox{in} \ \R^N,
\end{array} \right .
\end{align}
where $H$ and $u_0$ are some continuous functions. These nonlinear first order equations were first
introduced in classical mechanics, but find application in many other fields of mathematics. They also arise naturally in optimal control theory, due to the Dynamic Programming Principle. We refer to the book of Bardi and Capuzzo Dolcetta, \cite{BC}, for a thorough review on the subject. 

Equation \eqref{HJ} has been studied extensively in the past decades (see for example \cite{B, B2, B3, NJ, BSoug, Ishii}), where comparison, existence, uniqueness, large time asymptotics and qualitative properties have been researched. 

One particular important example of Hamilton-Jacobi equations is the so called \textit{eikonal equation}. There are several equivalent ways to state the equation, one of its most common form is 
\begin{align}\label{eik}
|Du|=f(x),\quad x\in\Omega,
\end{align}
for some open set $\Omega\ss\R^N$. This equation is related to many classical problems, for example geometrical optics and wave propagation, see for example \cite{Arnold}.

The aim of this paper is to generalize Hamilton-Jacobi equations and in particular \eqref{eik} to the fractional setting. More precisely, we want to study Hamilton-Jacobi equations where the dependence of the gradient is replaced by a nonlocal operator playing the same role. Next we motivate the definition of a nonlocal gradient.

Let $0<s<1$ and denote $(-\Delta)^s$ the fractional laplacian, that is
\[
(-\Delta)^su=\text{P.V.}c_{N,s}\int_{\R^N}\frac{u(x)-u(y)}{|x-y|^{N+2s}}dy,
\]
where $c_{N,s}$ is a  normalization constant taken so that $(-\Delta)^su\to-\Delta u$ for smooth functions $u$. Define also the operator $\cB$ is defined by 
\begin{align}\label{CB}
\mathcal B(u,u)=c_{N,s}\int_{\R^n}\frac{|u(x)-u(y)|^2}{|x-y|^{N+2s}}dy.
\end{align}

Geometrically, $\cB$ comes from a formulation of the heat flow of harmonic maps with free boundary written in nonlocal form, namely 
\[
\partial_t u +(-\Delta)^s u \perp T_{u(x)} \mathbb S^{m-1}.
\]
As described in \cite{Millot-Sire,DLP, DR2, DR2}, harmonic maps with free boundary can be reformulated in a very natural way as harmonic maps with respect to an $\dot H^{1/2}$-energy, hence leading to a nonlocal system of elliptic equations, namely
\[
u_t+(-\Delta)^s u=u\mathcal B(u,u),  \ \ \text{in}\ \R^N\times(-1,1).
\]

This type of operator also appears naturally in other areas of mathematics. For example, in probability this operator is known as the \textit{carr\'e du champ}, see the classic work by Bakry and \'Emery \cite{BE} and also \cite{BGL}. See also \cite{SWZ} for a connection of the Bakry-\'Emery theory of curvature dimension inequalities and the fractional laplacian. It also appears naturally when studying Dirichlet forms, see for example \cite{MR}.

Moreover since the operator $\cB$ is the \textit{carr\'e du champ} operator associated to the fractional laplacian, it appears in routine computations, for example 
\begin{align}\label{product}
\Delta^s(uv)=u \Delta^sv + v \Delta^su + 2\cB(u,v).
\end{align}
where we have set $-(-\Delta)^s=\Delta^s$ and also appears in standard integration by parts
\begin{align}\label{byparts}
\int_{\R^N}v \Delta^sudx=-\int_{\R^N}\cB(u,v)(x)dx,
\end{align}

Since $\cB$ appears in the product rule \eqref{product} and integration by parts formula \eqref{byparts}, it comes up in many applications, for example in the fractional Pohozaev identity \cite{RS2}, boundary behaviour \cite{RS}, Harnack inequalities \cite{DQT} and so on.

Let us go back to Hamilton-Jacobi equations. First note that the operator $\cB$ given by \eqref{CB} satisfies
\[
\cB(u,v)\to Du\cdot Dv,\quad s\to1
\]
and in particular
\[
D_s u\coloneqq \sqrt{\cB(u,u)}\to |Du| \quad s\to1.
\]

Therefore, a natural extension of \eqref{eik} would be the \textit{fractional eikonal equation} given by
\begin{align}\label{feik}
D_s u=f(x),\quad x\in\Omega.
\end{align}
More generally, we are interested in studying Hamilton-Jacobi equations of the form
\begin{align}\label{HJ1}
H(x,u,D_s u)=0 \quad x\in\Omega
\end{align} 
or a time dependent version of this equation.

%

The standard framework to study equation \eqref{HJ1} in the local case is the notion of viscosity solutions. In the local case a smooth classical solution is trivially a viscosity solution and vice-versa, although in the nonlocal case this is not obvious. 

In a recent paper by Barrios and Medina (see \cite{BM}) they studied the equivalence of weak and viscosity solutions of
\[
(-\Delta)^s_pu-H(x,u,D_s^p u)=0, \quad x\in\Omega,
\] 
where $(-\Delta)^s_p$ is the fractional $p$-laplacian, $\Omega$ is a bounded subset of $\R^N$ and 
\[
D_s^p u\coloneqq\int_{\R^N}\frac{|u(x)-u(y)|^p}{|x-y|^{N+2s}}dy.
\]
They proved that if
\[
H(x,t,\eta)\leq \gamma(|t|)\eta^{\frac{p-1}{p}}+\phi(x),
\]
for $\eta\geq 0$ and $\phi\in L^\infty_{loc}(\Omega)$ then any viscosity super solution is a weak solution.

Their result also holds for the equation
\begin{align}\label{sigma}
\sigma(-\Delta)^s_pu+H(x,u,D_s^p u)=0, \quad x\in\Omega
\end{align}
for any $\sigma>0$. In the particular case $p=2$, their result implies that whenever the Hamiltonian grows at most linearly on the nonlocal gradient then viscosity solutions are weak. This implies that smooth viscosity solutions are classical, which gives us a way of studying these equations under the viscosity framework by adding a diffusion term. 

More precisely, we can study equation \eqref{feik} by adding $\sigma(-\Delta)^s$ and considering the limit as $\sigma\to 0$, which is the original vanishing viscosity method. The same idea applies to equation \eqref{HJ1} if the Hamiltonian has a power like growth, since we can study the equation
\begin{align*}
\sigma(-\Delta)^su+\tilde H(x,u,D_s u)=0.
\end{align*}
where $\tilde H=H^{1/\tilde p}$ and $\tilde p$ is a sufficiently large odd power.

The aim of this paper is to prove the comparison principle for fractional Hamilton-Jacobi equations. Our result applies to a wider class of Hamilton-Jacobi equations, namely it applies to
\begin{align}\label{LHJ1s}
\sigma(-\Delta)^s u+H(x,D_su))=0,\quad x\in\Omega,
\end{align}
and 
\begin{align}\label{LHJ2s}
\sigma(-\Delta)^s u+H(x,\cB(u,h))=0, \quad x\in\Omega.
\end{align}
for $\sigma\geq 0$. We note that in the linear case 
\[
(-\Delta)^s u+\cB(u,h)=0, \quad x\in\Omega.
\]
there is a recent result by Topp and the author where the comparison principle is proven for \textit{weak} solutions under the additional hypothesis
\begin{align}\label{h}
\osc\limits_{\R^N}h<1.
\end{align}
This condition seems natural, since otherwise the full operator $(-\Delta)^s(\cdot)+\cB(\cdot,h)$ fails to satisfy the maximum principle. See the discussion after Corollary \ref{CPCH}.

Maximum principles and comparison principles is a subject that has been studied thoroughly in the past in several different contexts and equations. There is a vast literature in both local and nonlocal setting for both linear and fully nonlinear equations. We refer to the classic work by Crandall and Ishii \cite{CIs} and Ishii \cite{Is1} in the context of local equations. For nonlocal equations we refer to the work of Barles and Imbert \cite{BI}, Barles,  Chasseigne and Imbert \cite{BCI} and Jakobsen and Karlsen \cite{JK}. For nonlocal p-laplace operators we refer to the paper by Chasseigne and Jakobsen \cite{CJ} and the references therein

Next we state the main results of this article.

\begin{theorem}\label{sconvex}
Let $\Omega\subset\R^N$ be a bounded domain and $\sigma\geq 0$. Let $u_1, u_2\in C(\bar\Omega)\cap \text{B}(\R^N)$ be respectively sub and super solutions of \eqref{LHJ1s} in $\Omega$ with $u_1\leq u_2$ in $\Omega^c$. Suppose $H$ satisfies  hypotheses \emph{H1}, \emph{H2} and furthermore assume that the function $\bar u\rightarrow H(x,\bar u)$ is convex for each $x\in\Omega$. 

Then $u_1\leq u_2$ in $\Omega$.
\end{theorem}
The same result applies if $u_1, u_2$ are sub and super solution of \eqref{LHJ2s} and if $H$ satisfies (H1h) and (H2h) instead of (H1) and (H2).

We point out that Theorem \ref{sconvex} holds in the case $\sigma=0$. To the best of the authors knowledge this is the first result regarding comparison principles for purely nonlocal Hamilton-Jacobi equations and its new even in the case $\sigma>0$, since we are treating general nonlinearities.

It is possible to remove hypothesis (H2) ((H2h) resp.) and the convex hypothesis on $H$ in Theorem \ref{sconvex} whenever we are in presence of a proper term.

\begin{theorem}\label{sproper}
Let $\Omega\subset\R^N$ be a bounded domain. Fix $\mu>0$ and let $u_1, u_2\in C(\bar\Omega)\cap \text{B}(\R^N)$ be respectively sub and super solutions of 
\begin{align}\label{LHproper1}
\mu u+H(x,D_s u)=0
\end{align}
or
\begin{align}\label{LHproper2}
\mu u+H(x,\cB(u,h))=0
\end{align}
with $u_1\leq u_2$ in $\Omega^c$. Furthermore assume that hypothesis \emph{(H1)} holds in case of \eqref{LHproper1} and hypothesis \emph{(H1h)} holds  in case of \eqref{LHproper2}.

Then $u_1\leq u_2$ in $\Omega$.
\end{theorem}

The last result we present is the comparison principle for sub and super solutions of the parabolic equations
\begin{align}\label{LHJ1sp}
u_t+\sigma(-\Delta)^s u+H(x,t,D_s u)=0,\quad x\in\Omega\times(0,T),
\end{align}
and 
\begin{align}\label{LHJ2sp}
u_t + \sigma(-\Delta)^s u+H(x,t,\cB(u,h))=0, \quad x\in\Omega\times(0,T).
\end{align}
for $\sigma\geq 0$.
As in the local case the term $u_t$ acts as a proper term and therefore the hypotheses on $H$ are in the spirit of Theorem \ref{sproper}.

\begin{theorem}\label{spar}
Let $\sigma\geq 0$ and $u_1, u_2\in C(\bar\Omega[0,T])\cap B(\R^N)$ be respectively sub and super solution of \eqref{LHJ1sp} (\eqref{LHJ2sp} resp.) with $u_1\leq u_2$ in $\Omega^c\times(0,T)\cup \R^N\times\{0\}$.  Assume that $H$ satisfies \emph{(H3)} (\emph{(H3h)} resp.). 

Then $u_1\leq u_2$ in $\Omega\times(0,T)$.
\end{theorem}

We point out that most of the results presented in this work can be extended to more general operators of the form
\[
L_{K_1}u(x)=\text{P.V.}\int_{\R^N}(u(y)-u(x))K_1(x-y)dy.
\]
\[
\cB_K(u,v)(x)=\frac{1}{2}\int_{\R^N}(u(y)-u(x))(v(y)-v(x))K_2(x-y)dy.
\]
and the associated nonlocal gradient $D_s^{K_2}u=\sqrt{\cB_{K_2}}(u,u)$ under the hypothesis that
\[
\frac{\lambda c_{N, s}}{|z|^{N + 2s_i}} \leq K_{i}(z) \leq \frac{\Lambda c_{N, s_i}}{|z|^{N + 2s}}.
\]
for $i=1,2$ and $0<s_i<1$. Note we do not need to require that $s_1\geq s_2$, since the proof is solely based on the Hamiltonian part of the operator.

The paper is organized as follows. In Section \ref{SecPer} we give the precise definition of viscosity solution and the the hypotheses on $H$. In Section \ref{SecComp} we prove Theorem \ref{sconvex} and Theorem \ref{sproper} in the absence of diffusion ($\sigma=0$), as the diffusion is an innocuous term that respects the comparison property. Finally in Section \ref{SecPar} we prove Theorem \ref{spar} again in the particular case $\sigma=0$.


\section{Preliminaries}\label{SecPer}
 
In this section we state the main hypotheses on the Hamiltonian $H$ as well as giving the precise definition of viscosity sub and super solution.

We define the weight 
\[
w(x)=\frac{1}{1+|x|^{N+2s}},
\]
and the space $L^1_w(\R^N)$ to be set of functions such that  $\|uw\|_{L^1(\R^n)}<\infty$. 

From now on $h:\R\to\R$ denotes a smooth function in $L^1_w(\R^N)$ and for the rest of the section we consider a continuous Hamiltonian $H:\R^3\to\R$. Given $\Omega\subseteq\R^N$  we denote $B(\Omega)$ the set of bounded functions in $\Omega$.

We begin with the definition of viscosity sub and super solutions when we have a diffusion term.
\begin{definition}
A function $u:\R^n\to\R$ upper (lower) semicontinuous in $\bar\Omega$ is a viscosity sub solution (super solution) of \eqref{LHJ1s} if for any $\varphi\in C^2(\Omega)\cap L^1_w(\R^N)$ such that $u-\varphi$ has a global maximum at $x_0$ then
\[
\sigma(-\Delta)^s\varphi(x_0)+H(x_0,\varphi(x_0),D_s \varphi(x_0))\leq 0, (\geq 0)
\]

\end{definition}

As observed in \cite{BM} in the absence of the operator $\cB$ the definition can be stated for a larger class of sub and super solutions, by testing the the maximum (minimum) only locally. Furthermore these definitions coincide for general nonlocal diffusion operators, see for example \cite{Arisawa}. 

The definition of sub and super solutions of equation \eqref{LHJ2s}. Meanwhile the definition of sub and super solution for $s$-order Hamilton-Jacobi equations \eqref{HJ1} and 
\begin{align}\label{HJ2}
H(x,\cB(u,h))=0,\quad\text{in }\Omega.
\end{align}
are the same with the minor difference that we require the test function to be only $C^1(\Omega)\cap L^1_w(\R^N)$. 

We point out that the definition in the parabolic case is analogous and will be used in Section \ref{SecPar}. We state it next for completeness.

\begin{definition}
A function $u:\R^n\to\R$ upper (lower) semicontinuous in $\bar\Omega$ is a viscosity sub solution of \eqref{LHJ1sp} if for any $\varphi\in C^{2,1}(\Omega\times(0,T))\cap L^1_w(\R^N)$ such that $u-\varphi$ has a global maximum at $(x_0,t_0)$ then
\[
\varphi_t(x_0,t_0) +\sigma(-\Delta)^s\varphi(x_0,t_0)+H(x_0,t_0,\varphi(x_0,t_0),D_s \varphi(x_0,t_0))\leq 0,
\]

\end{definition}
The definition of sub and super solutions of equation \eqref{LHJ2sp} is analogous. Again, the definition of sub and super solution for $s$-order Hamilton-Jacobi equations \eqref{PHJ1} and \eqref{PHJ2} are the same with the minor difference that we require the test function to be only $C^{1,1}(\Omega\times(0,T))\cap L^1_w(\R^N)$.

Next we state the hypotheses on the Hamiltonian $H$. The first hypothesis is a modulus of continuity for $H$.
\begin{itemize}
\item[(H1)] We say that that the Hamiltonian $H:\R^2\to\R$ satisfies hypothesis (H1) if
there exists a modulus of continuity $\omega_1$ such that 
\begin{align*}
|H(x, D_s v(x))-H(y,D_s v(x))|\leq\omega_1(|x-y|(1+\cB(v)(x))),
\end{align*}
for all $x,y\in\Omega$ and $v\in C^1(\Omega)\cap L^1_w(\R^N)$. 
\item[(H1h)] We say that that the Hamiltonian $H:\R^2\to\R$ satisfies hypothesis (H1h) if there exists a modulus of continuity $\tilde\omega_1$ such that 
\begin{align*}
|H(x,\cB(v,h_1)(x))-H(y,\cB(v,h_2)(y))|\leq\tilde\omega_1(|x-y|+|\cB(v,h_1)(x)-\cB(v,h_2)(y)|),
\end{align*}
for all $x,y\in\Omega$ and $v\in C^1(\Omega)\cap L^1_w(\R^N)$.
\end{itemize}
The next hypothesis is crucial for the comparison principle in the context of convex Hamiltonians, as it guarantees the existence of a strict sub solution. We state it as its needed for Theorem \ref{sconvex}. 
\begin{itemize}
\item[(H2)] We say that that the Hamiltonian $H:\R^2\to\R$ satisfies hypothesis (H2) if, there exists $\rho_1\in C(\bar\Omega)\cap C^2(\Omega)\cap \text{B}(\R^N)$ such that $\rho_1\leq u_1$ in $\R^N$ and 
\[
\sup\limits_{x\in\Omega'}\sigma(-\Delta)^s\rho_1+ H(x,D_s\rho_1(x))<0,  \forall \Omega'\subset\subset\Omega.
\]
\item[(H2h)] We say that that the Hamiltonian $H:\R^2\to\R$ satisfies hypothesis (H2h) if, there exists $\rho_2\in C(\bar\Omega)\cap C^2(\Omega)\cap \text{B}(\R^N)$ such that $\rho_2\leq u_1$ in $\R^N$ and 
\[
\sup\limits_{x\in\Omega'}\sigma(-\Delta)^s\rho_2+ H(x,\cB(\rho_2,h)(x))<0,  \forall \Omega'\subset\subset\Omega.
\]
\end{itemize}
\begin{remark}
Hypothesis \emph{(H2)} (resp. \emph{(H2h)} will be used in the particular case $\sigma=0$, see Theorem \ref{CPC} for more details. In this case one can lower the smoothness hypothesis on $\rho_1$ and $\rho_2$ and just require $\rho_i\in C(\bar\Omega)\cap C^1(\Omega)\cap \text{B}(\R^N)$.
\end{remark}
We will need a particular type of modulus of continuity in the parabolic case in order to prove the comparison theorem.
\begin{itemize}
\item[(H3)] We say that that the Hamiltonian $H:\R^2\to\R$ satisfies hypothesis (H3) if there exists a modulus of continuity $\omega_1$ such that 
\begin{align*}
|H(x,t,D_s v(x))-H(y,s,D_s v(x))|\leq\omega_1(|t-s|, |x-y|(1+\cB(v)(x))),
\end{align*}
for all $x,y\in\Omega$ $t,s\in[0,T]$ and $v\in C^1(\Omega\times[0,T])\cap L^1_w(\R^N)$. 
\item[(H3h)] We say that that the Hamiltonian $H:\R^2\to\R$ satisfies hypothesis (H3h) if there exists a modulus of continuity $\tilde\omega_1$ such that 
\begin{align*}
|H(x,t,\cB(v,h_1)(x))-H(y,s,\cB(v,h_2)(y))|\leq\tilde\omega_1(|t-s|,|x-y|+|\cB(v,h_1)(x)-\cB(v,h_2)(y)|),
\end{align*}
for all $x,y\in\Omega$ $t,s\in[0,T]$ and $v\in C^1(\Omega\times[0,T])\cap L^1_w(\R^N)$. 
\end{itemize}

hypothesis

\section{Comparison Principles}\label{SecComp}

In this section we provide the proof of the Theorem \ref{sconvex} and Theorem \ref{sproper}. We will first prove the result in the absence of diffusion ($\sigma=0$), since the general proof follows the same ideas.

In the case of convex Hamiltonians we have the following theorem.

\begin{theorem}[Comparison Principle for convex Hamiltonians]\label{CPC}
Let $\Omega\subset\R^N$ be a bounded domain. Let $u_1, u_2\in C(\bar\Omega)\cap \text{B}(\R^N)$ be respectively sub and super solutions of \eqref{HJ1} in $\Omega$ with $u_1\leq u_2$ in $\Omega^c$. Suppose $H$ satisfies  hypotheses \emph{H1}, \emph{H2} and furthermore assume that the function $\bar u\rightarrow H(x,\bar u)$ is convex for each $x\in\Omega$.

Then $u_1\leq u_2$ in $\Omega$.
\end{theorem}

\begin{proof}
The proof is based on the comparison principle for viscosity solutions of first order equations, see for example \cite{Barles}. 

Let $t\in(0,1)$ and consider the auxiliary function 
\[
u^t(x)=tu_1(x)+(1-t)\rho_1(x).
\]
Note that any test function of $\varphi$ of $u^t$ at $x_0$ can be rewritten as 
\[
\varphi=t\tilde\varphi+(1-t)\rho_1
\]
where $\tilde\varphi$ is a test function of $u_1$ at the same point $x_0$. Now, thanks to the convexity hypothesis we have 
\begin{align*}
H(x_0,D_s \varphi(x_0))&=H(x_0,D_s(t\tilde\varphi+(1-t)\rho_1)(x_0))\\
&\leq tH(x_0,D_s\tilde\varphi(x_0))+(1-t)H(x_0,D_s\rho_1(x_0),
\end{align*}
and since $u_1$ is a sub solution and $\tilde\varphi$ is a test function for $u_1$ we deduce
\[
H(x_0,D_s \varphi(x_0))\leq (1-t)H(x_0,D_s\rho_1(x_0)).
\]
This implies $u^t$ is a sub solution 
\begin{align}\label{aux}
H(x,D_su^t)=f^t(x),
\end{align}
where $f^t(x)=(1-t)H(x,D_s\rho_1(x))$ is a continuous function, since $\rho_1\in C(\bar\Omega)\cap C^1(\Omega)\cap \text{B}(\R^N)$. We also have that $u^t\leq u_2$ in $\Omega^c$. 

Next we prove that $u^t\leq u_2$ in $\R^N$ for all $t\in(0,1)$. We proceed by contradiction and assume that for some $\bar t$ fixed there exists $\delta>0$ such that
\[
\sup\limits_{\Omega}(u^{\bar t}-u_2)=\delta.
\]
We now proceed we standard method of doubling variables. Consider the function
\[
\Phi_\varepsilon(x,y)=u^{\bar t}(x)-u_2(y)-\eta\left(\frac{|x-y|}{\varepsilon^{1/2}}\right)
\]
where $\eta$ is a smooth positive function given by
\begin{align}\label{eta}
\eta(r)=\begin{cases} r^2 &\mbox{if } |r|\leq (3\|u_2\|_{L^\infty(\R^N)|})^{1/2} \\ 
\bar\eta & \mbox{if } |r|>1 \end{cases}
\end{align}
and $\bar \eta $ is a smooth symmetric increasing function growing like $r^\alpha$ when $r\to\infty $ with $\alpha< s$. Define 
\[
\Omega'=\{x\in\Omega\ \text{s.t. } (u^{\bar t}-u_2)>\delta/2\}.
\]
Since $u_2\geq u ^{\bar t}$ in $\Omega^c$ it is clear that $\bar\Omega'\subset\Omega$. Let $(x_\varepsilon,y_\varepsilon)$ be a maximum point of $\Phi_\varepsilon$ on $\bar\Omega'\times \bar\Omega'$. Note that since $\eta\geq0$ we have
\[
\Phi_\varepsilon(x_\varepsilon,y_\varepsilon)\leq \Phi_\varepsilon(x_\varepsilon,x_\varepsilon),
\]
which implies 
\[
\eta\left(\frac{|x-y|}{\varepsilon^{1/2}}\right)\leq u_2(x_\varepsilon)-u_2(y_\varepsilon).
\]
Using the definition of $\eta$ we deduce
\[
|x-y|^2\leq 2\|u\|_{L^{\infty}(\R^N)}\varepsilon.
\]
We deduce then that $|x_\varepsilon-y\varepsilon|\to 0$ as $\varepsilon\to 0^+$. This combined with the continuity of $u_2$ implies
\begin{align}\label{rate}
\frac{|x_\varepsilon-y_\varepsilon|^2}{\varepsilon}\to 0,\quad \text{as }\varepsilon\to 0^+.
\end{align}
Observe also that by the contradiction hypothesis and the definition of $\Omega'$
\[
\Phi_\varepsilon(x_\varepsilon,y_\varepsilon)\geq\sup\limits_{x\in\Omega'}\Phi_\varepsilon(x,x)=\delta.
\]
We claim that $(x_\varepsilon,y_\varepsilon)\in\Omega'\times\Omega'$ for $\varepsilon$ small. If not, then either $x_\varepsilon$ or $y_\varepsilon$ belongs to $\partial\Omega'$, therefore
\begin{align}
\frac{\delta}{2}u^{\bar t}(x_\varepsilon)-u_2(\varepsilon)&\geq\Phi_\varepsilon(x_\varepsilon, y_\varepsilon)+u_2(y_\varepsilon)-u_2(x_\varepsilon)\\
&\geq\delta-\omega(|x_\varepsilon-y_\varepsilon|),
\end{align}
where $\omega$ is a modulus of continuity of $u_2$, which leads to a contradiction for $\varepsilon$ small enough since $x_\varepsilon\to y_\varepsilon$ as $\varepsilon\to 0^+$.

Let
\[
\varphi_1(x)=u_2(y_\varepsilon)+\eta\left(\frac{|x-y_\varepsilon|}{\varepsilon^{1/2}}\right).
\]

Observe that since $x_\varepsilon$ is a maximum of $\Phi(x,y_\varepsilon)$, we have that $x_\varepsilon$ is a maximum of $u^{\bar t}(x)-\varphi_1(x)$. Moreover by taking small enough we can guarantee that the maximum is global. Since $u^{\bar t}$ is a viscosity sub solution of \eqref{aux} we obtain 
\[
H(x_\varepsilon, D_s\varphi_1(x_\varepsilon))-f^{\bar t}(x_\varepsilon)\leq0.
\] 
Define now
\[
\varphi_2(y)=u^{\bar t}(x_\varepsilon)-\eta\left(\frac{|x_\varepsilon-y|}{\varepsilon^{1/2}}\right),
\]
and reasoning as before we have that $y_\varepsilon$ is a minimum of $u_2(y)-\varphi_2(y)$ for $\varepsilon$ small. Since $u_2$ is a viscosity super solution of \eqref{HJ1} we obtain
\[
H(y_\varepsilon,D_s\varphi_2(y_\varepsilon))\geq 0.
\] 
Furthermore we have
\begin{align*}
(D_s \varphi_1)^2(x_\varepsilon)&=\frac{c_{N,s}}{2}\int_{\R^N}\frac{(\varphi_1(x_\varepsilon)-\varphi_1(y))^2}{|x_\varepsilon-y|^{N+2s}}dy\\
&=\frac{c_{N,s}}{2}\int_{\R^N}\frac{(\varphi_1(x_\varepsilon)-\varphi_1(y))^2}{|y-x_\varepsilon|^{N+2s}}dy\\
&=\frac{c_{N,s}}{2}\int_{\R^N}\frac{(\varphi_1(x_\varepsilon)-\varphi_1(z+x_\varepsilon))^2}{|z|^{N+2s}}dy\\
&=\frac{c_{N,s}}{2}\int_{\R^N}\frac{(\eta(|x_\varepsilon-y_\varepsilon|\varepsilon^{-1/2})-\eta(|z+x_\varepsilon-y_\varepsilon|\varepsilon^{-1/2}))^2}{|z|^{N+2s}}dy,
\end{align*}
and on the other hand
\begin{align*}
(D_s \varphi_2)^2(y_\varepsilon)&=\frac{c_{N,s}}{2}\int_{\R^N}\frac{(\varphi_2(y_\varepsilon)-\varphi_2(y))^2}{|y_\varepsilon-y|^{N+2s}}dy\\
&=\frac{c_{N,s}}{2}\int_{\R^N}\frac{(\varphi_2(y_\varepsilon)-\varphi_2(y))^2}{|y-y_\varepsilon|^{N+2s}}dy\\
&=\frac{c_{N,s}}{2}\int_{\R^N}\frac{(\varphi_2(y_\varepsilon)-\varphi_1(z+y_\varepsilon))^2}{|z|^{N+2s}}dy\\
&=\frac{c_{N,s}}{2}\int_{\R^N}\frac{(\eta(|x_\varepsilon-y_\varepsilon|\varepsilon^{-1/2})-\eta(|z+x_\varepsilon-y_\varepsilon|\varepsilon^{-1/2}))^2}{|z|^{N+2s}}dy,
\end{align*}
and therefore we deduce $D_s\varphi_1(x_\varepsilon)=D_s\varphi_2(y_\varepsilon)$.

Thanks to hypothesis (H2), there is $\alpha>0$ such that
\[
H(x,D_s\rho_1)\leq -\alpha,
\]
and since $1>\bar t>0$ this implies 
\[
f^{\bar t}(x)=(1-t)H(x,D_s\rho_1)\leq \bar\alpha<0.
\]
Combining the previous inequalities we get and using hypothesis \textit{i.-}
\begin{align}\label{contradiction}
\bar\alpha &\leq H(y_\varepsilon,D_s\varphi_2(y_\varepsilon))-H(x_\varepsilon,D_s\varphi_1(x_\varepsilon))\notag\\
&=H(y_\varepsilon,D_s\varphi_1(x_\varepsilon))-H(x_\varepsilon,D_s\varphi_1(x_\varepsilon))\notag\\
&\leq \omega_1(|x_\varepsilon-y_\varepsilon|(1+D_s\varphi_1(x_\varepsilon)).
\end{align}
In order to arrive to a contradiction we need to prove that $|x_\varepsilon-y_\varepsilon|D_s\varphi_1(x_\varepsilon)$ converges to 0 as $\varepsilon\to 0^+$. Denote $w=x_\varepsilon-y_\varepsilon$, we have
\begin{align*}
(D_s\varphi_1)^2(x_\varepsilon)&=\frac{c_{N,s}}{2}\int_{\R^N}\left(\eta\left(\frac{|w|}{\varepsilon^{1/2}}\right)-\eta\left(\frac{|z+w|}{\varepsilon^{1/2}}\right)\right)^2\frac{1}{|z|^{N+2s}}dz\\
&=I_1+I_2,
\end{align*}
where
\begin{align*}
I_1&=\frac{c_{N,s}}{2}\int_{B_2^c}\left(\eta\left(\frac{|w|}{\varepsilon^{1/2}}\right)-\eta\left(\frac{|z+w|}{\varepsilon^{1/2}}\right)\right)^2\frac{1}{|z|^{N+2s}}dz\\
I_2&=\frac{c_{N,s}}{2}\int_{B_2}\left(\eta\left(\frac{|w|}{\varepsilon^{1/2}}\right)-\eta\left(\frac{|z+w|}{\varepsilon^{1/2}}\right)\right)^2\frac{1}{|z|^{N+2s}}dz
\end{align*}
Let us analyze $I_1$. We have $\eta(|w+z|)\geq\eta(|w|)\geq0$ for all $|z|>2$, therefore
\begin{align*}
I_1&=\frac{c_{N,s}}{2}\int_{B_2^c}\left(\eta\left(\frac{|w|}{\varepsilon^{1/2}}\right)-\eta\left(\frac{|z+w|}{\varepsilon^{1/2}}\right)\right)^2\frac{1}{|z|^{N+2s}}dz\\
&\leq \frac{c_{N,s}}{2}\int_{B_2^c}\eta^2\left(\frac{|z+w|}{\varepsilon^{1/2}}\right)\frac{1}{|z|^{N+2s}}dz\\
&\leq C\frac{c_{N,s}}{\varepsilon^\alpha}\int_{B_2^c}\frac{\eta^2(|z+w|)}{|z|^{N+2s}}dz\\
&=C\frac{c_{N,s}}{\varepsilon^\alpha},
\end{align*}
where $C$ is a universal constant. 

For $I_2$ we proceed with a Taylor expansion for $\eta$ to deduce
\begin{align*}
I_2&=\frac{c_{N,s}}{2}\int_{B_2}\left(\eta\left(\frac{|w|}{\varepsilon^{1/2}}\right)-\eta\left(\frac{|z+w|}{\varepsilon^{1/2}}\right)\right)^2\frac{1}{|z|^{N+2s}}dz\\
&= \frac{c_{N,s}}{2\varepsilon^s}\int_{B_2}\left(\eta\left(\frac{|w|}{\varepsilon^{1/2}}\right)-\eta\left(\left|z+\frac{w}{\varepsilon^{1/2}}\right|\right)\right)^2\frac{1}{|z|^{N+2s}}dz\\
&=\frac{c_{N,s}}{2\varepsilon^s}\int_{B_2}\frac{(D\eta(w\varepsilon^{-1/2})\cdot z+\text{o}(|z|^2))^2}{|z|^{N+2s}}dz\\
&\leq C \frac{c_{N,s}}{\varepsilon^s}\int_{B_2}\frac{|D\eta(w\varepsilon^{-1/2})|^2|z|^2+|D\eta(w\varepsilon^{-1/2})|\text{o}(|z|^3)+\text{o}(|z|^4)}{|z|^{N+2s}}dz.
\end{align*}
for $C$ some universal constant. Since $D\eta (w\varepsilon^{-1/2})=2w\varepsilon^{-1/2}$ we deduce
\[
I_2\leq C \frac{c_{N,s}}{2\varepsilon^s}\left(\frac{|w|^2}{\varepsilon}+\frac{|w|}{\varepsilon^{-1/2}}+1\right).
\]
Recall that $|w|^2/\varepsilon\to 0$ as $\varepsilon\to 0$ (see \eqref{rate}) we get 
\[
I_2\leq C \frac{c_{N,s}}{2\varepsilon^s}\text{o(1)},
\]
and since $\alpha<s$ we get that $D_s\varphi_1(x_\varepsilon)$ is dominated by the bounds of $I_2$ that is
\begin{align}\label{controlgradnl}
D_s\varphi_1(x_\varepsilon)\leq \tilde{c}\frac{1}{\varepsilon^{s/2}}\text{o(1)},
\end{align}
where $\tilde c=Cc^{1/2}_{N,s}$. Finally, thanks to \eqref{rate} we have $|w|^2/\varepsilon\to 0$ and so we get
\begin{align*}
|w|D_s\varphi_1(x_\varepsilon)&\leq  \tilde c \frac{|w|}{\varepsilon^{s/2}}\\
&=\tilde c \frac{|w|^{s}}{\varepsilon^{s/2}}|w|^{1-s}\to 0\quad \text{as }\varepsilon \to 0,
\end{align*}
since  $s\in(0,1)$. This contradicts \eqref{contradiction} and finishes the proof.
\end{proof}

\begin{remark}
Observe that hypothesis \emph{(H2)} is trivially satisfied for Hamiltonians of the form
\[
H(x,D_s u)=|D_s u|^p-f(x),
\]
where $f(x)>0$. In particular there exists, for each $\sigma\geq0$ at most one viscosity solution of
\[
\sigma(-\Delta)^su+|D_s u|-f(x)=0.
\]
\end{remark}

We now proof the comparison principle for equations depending on the transport term $\cB(\cdot, h)$. The proof utilizes the same ideas as in Theorem \ref{CPC}, although we need to take special care in order to bound the error in the modulus of continuity.

\begin{corollary}\label{CPCH}
Let $\Omega\subset\R^N$ be a bounded domain. Let $u_1, u_2\in C(\bar\Omega)\cap \text{B}(\R^N)$ be respectively sub and super solutions of \eqref{HJ2} in $\Omega$ with $u_1\leq u_2$ in $\Omega^c$. Suppose $H$ satisfies  hypotheses \emph{(H1h)}, \emph{(H2h)} and furthermore assume that the function $\bar u\rightarrow H(x,\bar u)$ is convex for each $x\in\Omega$.

Then $u_1\leq u_2$ in $\Omega$.
\end{corollary}
\begin{proof}
We only provide the proof in the case $\sigma=0$, that is for equation \eqref{HJ2}. The proof is the same as in Theorem \ref{CPC}, up to the point of the definition of $\varphi_1,\varphi_2$, where we have 
\[
H(x_\varepsilon, \cB(\varphi_1,h)(x_\varepsilon))-f^{\bar t}(x_\varepsilon)\leq0, \quad \text{and } H(y_\varepsilon,\cB(\varphi_2,h))(y_\varepsilon)\geq 0,
\]
and 
\[
f^{\bar t}(x)=(1-t)H(x,\cB(\rho_2,h)).
\]
Now, 
\begin{align*}
\cB(\varphi_1,h)(x_\varepsilon)&= \frac{c_{N,s}}{2}\int_{\R^N}\frac{(\varphi_1(x_\varepsilon)-\varphi_1(y))(h(x_\varepsilon)-h(y))}{|x_\varepsilon-y|^{N+2s}}dy\\
&= \frac{c_{N,s}}{2}\int_{\R^N}\frac{(\varphi_1(x_\varepsilon)-\varphi_1(z+x_\varepsilon))(h(x_\varepsilon)-h(z+x_\varepsilon))}{|z|^{N+2s}}dz\\
&= \frac{c_{N,s}}{2}\int_{\R^N}\frac{(\eta(|x_\varepsilon-y_\varepsilon|\varepsilon^{-1/2})-\eta(|z+x_\varepsilon-y_\varepsilon|\varepsilon^{-1/2}))(h(x_\varepsilon)-h(z+x_\varepsilon))}{|z|^{N+2s}}dz,
\end{align*}
and on the other hand
\begin{align*}
\cB(\varphi_2,h)(y_\varepsilon)&=\frac{c_{N,s}}{2}\int_{\R^N}\frac{(\varphi_2(y_\varepsilon)-\varphi_2(y))(h(y_\varepsilon)-h(y))}{|y_\varepsilon-y|^{N+2s}}dy\\
&= \frac{c_{N,s}}{2}\int_{\R^N}\frac{(\varphi_2(y_\varepsilon)-\varphi_2(z+y_\varepsilon))(h(y_\varepsilon)-h(z+y_\varepsilon))}{|z|^{N+2s}}dz\\
&= -\frac{c_{N,s}}{2}\int_{\R^N}\frac{(\eta(|x_\varepsilon-y_\varepsilon|\varepsilon^{-1/2})-\eta(|z+x_\varepsilon-y_\varepsilon|\varepsilon^{-1/2}))(h(y_\varepsilon)-h(z+y_\varepsilon))}{|z|^{N+2s}}dz
\end{align*}
Thanks to hypothesis \textit{iii.-}, there is $\alpha>0$ such that
\[
H(x,\cB(\rho_2,h))\leq -\alpha,
\]
and since $1>\bar t>0$ this implies 
\[
f^{\bar t}(x)=(1-t)H(x,\cB(\rho_2,h))\leq -\bar\alpha<0.
\]
Combining the previous inequalities we get and using hypothesis \textit{i.-} 
\begin{align*}
\bar\alpha &\leq H(y_\varepsilon,\cB(\varphi_2,h)(y_\varepsilon))-H(x_\varepsilon,\cB(\varphi_1,h)(x_\varepsilon))\\
&\leq \tilde\omega_1(|x_\varepsilon-y_\varepsilon|+|\cB(\varphi_2,h)(y_\varepsilon)-\cB(\varphi_1,h)(x_\varepsilon)|).
\end{align*}
Now the H\"older inequality implies,
\begin{align*}
&\cB(\varphi_2,h)(y_\varepsilon)-\cB(\varphi_1,h)(x_\varepsilon)\\
&=\frac{c_{N,s}}{2}\int_{\R^N}\frac{(\eta(|x_\varepsilon-y_\varepsilon|\varepsilon^{-1/2})-\eta(|z+x_\varepsilon-y_\varepsilon|\varepsilon^{-1/2}))(h(x_\varepsilon)-h(z+x_\varepsilon)- h(y_\varepsilon)+h(z+y_\varepsilon))}{|z|^{N+2s}}dz\\
&\leq CD_s\varphi_1(x_\varepsilon)H^{1/2}
\end{align*}
where
\[
H=\int_{\R^N}\frac{(h(x_\varepsilon)-h(z+x_\varepsilon)- h(y_\varepsilon)+h(z+y_\varepsilon))^2}{|z|^{N+2s}}.
\]
Rewrite $H=H_1+H_2$, where
\begin{align*}
H_1&=\int_{B_1}\frac{(h(x_\varepsilon)-h(z+x_\varepsilon)- h(y_\varepsilon)+h(z+y_\varepsilon))^2}{|z|^{N+2s}}\\
H_2&=\int_{B_1^c}\frac{(h(x_\varepsilon)-h(z+x_\varepsilon)- h(y_\varepsilon)+h(z+y_\varepsilon))^2}{|z|^{N+2s}},
\end{align*}
and let us analyze the behavior of $H_1$ and $H_2$ as $\varepsilon\to 0$. We have
\begin{align*}
H_1&=\int_{B_1}\frac{(h(x_\varepsilon)-h(z+x_\varepsilon)- h(y_\varepsilon)+h(z+y_\varepsilon))^2}{|z|^{N+2s}}\\
&=\int_{B_1}\frac{1}{|z|^{N+2s}}\left( \int_0^1(1-t)(Dh(x_\varepsilon+tz)- Dh(y_\varepsilon+tx)\cdot zdt\right)^2 dz\\
&\leq C_1|x_\varepsilon-y_\varepsilon|^2\int_{B_1}\frac{|z|^2}{|z|^{N+2s}}dz\\
&\leq C_1|x_\varepsilon-y_\varepsilon|^2,
\end{align*}
where $C$ is a constant depends on $\|h\|_{C^2}$. On the other hand
\begin{align*}
H_2&=\int_{B_1^c}\frac{(h(x_\varepsilon)-h(z+x_\varepsilon)- h(y_\varepsilon)+h(z+y_\varepsilon))^2}{|z|^{N+2s}}\\
&\leq C_2|x_\varepsilon-y_\varepsilon|^2\int_{B_1^c}\frac{dz}{|z|^{N+2s}}\\
&\leq C_2|x_\varepsilon-y_\varepsilon|^2,
\end{align*}
where $C_2$ depends on $\|h\|_{L^\infty}$. We conclude then
\[
H\leq C|x_\varepsilon-y_\varepsilon|^2,
\]
and so we get
\[
|\cB(\varphi_2,h)(y_\varepsilon)-\cB(\varphi_1,h)(x_\varepsilon)|\leq CD_s\varphi_1(x_\varepsilon)|x_\varepsilon-y_\varepsilon|.
\]
We use the estimate \eqref{controlgradnl} to conclude
\[
|\cB(\varphi_2,h)(y_\varepsilon)-\cB(\varphi_1,h)(x_\varepsilon)|\leq C\frac{1}{\varepsilon^{s/2}}\text{o(1)}|x_\varepsilon-y_\varepsilon|\to 0 \quad \text{as }\varepsilon\to 0.
\]
At this point we finish the proof as in the previous case.
\end{proof}

\begin{remark}\label{laplace}
Theorem \ref{sconvex} is a corollary of Theorem \ref{CPC} and Corollary \ref{CPCH}, since we observe that 
\[
(-\Delta)^s\varphi_1(x_\varepsilon)=(-\Delta)^s\varphi_2(y_\varepsilon).
\]
This holds thanks to the special structure of the function $\eta$. 
\end{remark}

We now come back to the fact that the notion of viscosity solution and weak solutions may not coincide for general Hamiltonians. We borrow the example from \cite{Davila-Topp} and include the computation for completeness. 

In~\cite{Dyda} the author computes the fractional laplacian explicitly for the function $u(x) = (1 - x^2)_+^{1 + s}$, $x \in \R$, where he gets the identity
\[
(-\Delta)^s u(x) = c(1 - (1 + 2s)x^2)_+, \quad x \in (-1,1).
\]
	
Now, let us consider a transport term $h: \R \to \R_+$ smooth and bounded such that $h = 0$ in $(-1,1)$, then
\begin{align*}
B(h,u)(x) &=   \int_{\R} (u(y) - u(x))(h(y) - h(x)) |x - y|^{-(1 + 2s)}dy \\
&= - u(x) \int_{[-1,1]^c} h(y) |x - y|^{-(1 + 2s)}dy\\
&\leq 0.
\end{align*}
	
Therefore we get
\[
(-\Delta)^s u(x) + B(h,u)(x) \leq 0.
\]
in the region $1 - (1 + 2s)x^2 \leq 0$. On the other hand, if $1 > (1 + 2s)x^2$, then $u$ is uniformly bounded by below by some constant $a_s>0$, that is $u(x)\geq a_s$. Hence for $x$ such that $1 > (1 + 2s)x^2$ we have
\[
Lu\coloneqq(-\Delta)^s u(x) + B(h,u)(x) \leq c - a_s \int_{[-1,1]^c} h(y) |x - y|^{-(1 + 2s)}dy.
\]
We can take $h$ large enough outside $(-1,1)$ in order to get 
\[
(-\Delta)^s u(x) + B(h,u)(x) \leq 0 \quad \mbox{en} \ (-1,1),
\]
with $u = 0$ in $(-1,1)^c$. But $u(0) > 0$ and maximum principle fails.

This example gives us a classical solution that does not satisfy the maximum principle and in view of Theorem \ref{CPC} fails to be a viscosity solution. This appears to be natural given that in the previous example $\osc_{R^N}h>1$ and therefore we loose ellipticity of the operator $L$. 

In the special case where the Hamiltonian is proper the comparison principle holds without assuming convexity nor the existence of a strict sub solution, that is without assuming (H2) or (H2h).
\begin{theorem}\label{Hproper}
Let $\Omega\subset\R^N$ be a bounded domain. Fix $\mu>0$ and let $u_1, u_2\in C(\bar\Omega)\cap \text{B}(\R^N)$ be respectively sub and super solutions of 
\begin{align}\label{Hproper1}
\mu u+H(x,D_su)=0
\end{align}
or
\begin{align}\label{Hproper2}
\mu u+H(x,\cB(u,h))=0
\end{align}
with $u_1\leq u_2$ in $\Omega^c$. Furthermore assume that hypothesis \emph{(H1)} holds in case of \eqref{LHproper1} and hypothesis \emph{(H1h)} holds  in case of \eqref{LHproper2}. 

Then $u_1\leq u_2$ in $\Omega$.
\end{theorem}
\begin{proof}
The proof follows the same ideas as the proof of Theorem \ref{CPC}. Let
\[
\Phi_\varepsilon(x,y)=u_1(x)-u_2(y)-\eta\left(\frac{|x-y|}{\varepsilon^{1/2}}\right),
\]
where $\eta$ is defined in \eqref{eta} and let $(x_\varepsilon, y_\varepsilon)$ be a maximum point of $\Phi_\varepsilon$. Since both $u_1$ and $u_2$ are bounded we can assume, as in the proof of Theorem \eqref{CPC}, that this is a global maximum. 

From this point we proceed in the exact same manner to conclude that 
\begin{align}\label{rate2}
\frac{|x_\varepsilon-y_\varepsilon|^2}{\varepsilon}\to 0,\quad \text{as }\varepsilon\to 0^+.
\end{align}
This rate of convergence guarantees that
\begin{align}\label{Phito0}
\liminf\limits_{\varepsilon\to 0}\Phi_\varepsilon(x_\varepsilon,y_\varepsilon)\leq0.
\end{align}
Indeed, suppose first that the maximum point $(x_{\varepsilon_n},y_{\varepsilon_n})$ is achieved at the boundary for some sequence $\varepsilon_n\to 0$. If $x_{\varepsilon_n}\in\partial\Omega$ then 
\[
u_1(x_{\varepsilon_n})-u_2(y_{\varepsilon_n})\leq u_2(x_{\varepsilon_n})-u_2(y_{\varepsilon_n}),
\]
since $u_1\leq u_2$ in $\Omega^c$. Since \eqref{rate2} holds and $u_2$ is uniformly continuous then the right hand side of the previous equation goes to zero as $\varepsilon_n\to 0$. Observe now that
\[
\Phi_{\varepsilon_n}(x_{\varepsilon_n},y_{\varepsilon_n})\leq u_1(x_{\varepsilon_n})-u_2(y_{\varepsilon_n}),
\]
and therefore \eqref{Phito0} holds. An analogous statement can be made when $y_n\in\partial\Omega$. 

Let us verify \eqref{Phito0} in the case $(x_{\varepsilon_n},y_{\varepsilon_n})\in\Omega$ for all $\varepsilon$  small. Define 
\[
\varphi_1(x)=u_2(y_\varepsilon)+\eta\left(\frac{|x-y_\varepsilon|}{\varepsilon^{1/2}}\right).
\]
and
\[
\varphi_2(y)=u_1(x_\varepsilon)-\eta\left(\frac{|x_\varepsilon-y|}{\varepsilon^{1/2}}\right),
\]
As in the proof of Theorem \ref{CPC} it is direct to verify that $\varphi_i$ is a test function for $u_i$ for $i=1,2$, and therefore
\[
\mu u_1(x_\varepsilon)+H(x_\varepsilon,D_s\varphi_1(x_\varepsilon))\leq 0 \leq\mu u_2(y_\varepsilon)+H(y_\varepsilon,D_s\varphi_2(y_\varepsilon))
\]
Furthermore we also have
\[
D_s\varphi_1(x_\varepsilon)=D_s\varphi_2(y_\varepsilon),
\]
and therefore hypothesis \textit{i} implies
\[
\mu(u_1(x_\varepsilon)-u_2(y_\varepsilon))\leq \omega_1(|x_\varepsilon-y_\varepsilon|(1+D_s\varphi_1(x_\varepsilon)), 
\]
which implies
\[
\Phi_{\varepsilon_n}(x_{\varepsilon_n},y_{\varepsilon_n})\leq \frac{1}{\mu}\omega_1(|x_\varepsilon-y_\varepsilon|(1+D_s\varphi_1(x_\varepsilon)).
\]
The right hand side of the previous equation goes to zero as $\varepsilon$ as seen in the proof of Theorem \ref{CPC} and therefore \eqref{Phito0} holds. 

Finally note that
\[
\max\limits_{x\in\bar\Omega}(u_1-u_2)\leq\Phi_\varepsilon(x_\varepsilon,y_\varepsilon), 
\]
which together with \eqref{Phito0} implies the desired result.

In the case $u_1$ and $u_2$ are sub and super solutions of \eqref{Hproper1} and \eqref{Hproper2} the proof follows the same ideas and uses the fact that 
\[
\tilde\omega_1(|x_\varepsilon-y_\varepsilon|+|\cB(\varphi_2,h)(y_\varepsilon)-\cB(\varphi_1,h)(x_\varepsilon)|)
\]
can be bounded as in Theorem \ref{CPC}.
\end{proof}

Observe, as in the case of Theorem \ref{sconvex}, that Theorem \ref{sproper} is a direct consequence of Theorem \ref{Hproper} and Remark \ref{laplace}.

\section{Comparison Principle in the parabolic case}\label{SecPar}
This last section is devoted to the proof of the comparison principle in the parabolic case. The proof does not require the Hamiltonian to be convex, since in the parabolic case $u_t$ plays the role of a proper term. As in the previous section we will only present the proof in the case $\sigma=0$ since the proof does not differ in the presence of diffusion.

\begin{theorem}\label{TeoP}
Let $\Omega\subset\R^N$ be a bounded domain and $u_1, u_2\in C(\bar\Omega\times[0,T])\cap \text{B}(\R^N\times[0,T])$ be respectively sub and super solutions of 
\begin{align}\label{PHJ1}
u_t+H(x,t,D_su(x))=0,\quad \text{in } \Omega\times(0,T),
\end{align}
or
\begin{align}\label{PHJ2}
u_t+H(x,t,\cB(u,h)(x))=0,\quad \text{in } \Omega\times(0,T),
\end{align}
with $u_1\leq u_2$ in $\Omega^c\times(0,T)\cup \R^N\times\{0\}$. Assume that $H$ satisfies \emph{(H3)} in case of \eqref{PHJ1} or \emph{(H3h)} in case of \eqref{PHJ2}.

Then $u_1\leq u_2$ in $\Omega\times(0,T)$.

\end{theorem}
Before the proof of the theorem we state a classic result of stability of viscosity solutions. 

\begin{lemma}
Let $u_1$ and $u_2$ be respectively sub and super solutions of \eqref{PHJ1} or \eqref{PHJ2}. Then
they are sub and super solutions in $\Omega\times(0,T]$. More precisely the viscosity
inequalities hold if the maximum or minimum points are on $\Omega\times\{T\}$
\end{lemma}
The proof of the previous lemma is identical to Lemma 5.1 in \cite{Barles}. 

\begin{proof}
The proof follows the same ideas as the proof of Theorem \ref{Hproper}, the main difference being that we need to include the time variable in the penalization scheme.
Denote $\Omega_T=\Omega\times (0,T)$ and let
\[
M=\max\limits_{\bar\Omega_T}(u_1-u_2)
\]

Assume that the result does not hold, that is $M>0$ and also assume without loss of generality that $u_1$ is a strict sub solution of the problem. This can easily be achieved by taking $u_1^\rho=u_1-\rho t$ for $\rho>0$ and observing that 
\[
u_1^\rho+H(x,t,D_su_1^\tau)\leq -\rho<0.
\]
From now on we just assume $u_1$ is a strict sub solution and drop the dependence on $\rho$.

Let
\[
\Phi(x,t,y,\tau)=u_1(x,t)-u_2(y,\tau)-\eta\left(\frac{|x-y|}{\varepsilon_1^{1/2}}\right)-\frac{(t-\tau)^2}{\varepsilon_2}
\]
where $\eta$ is given by a slight variation of \eqref{eta}, more precisely is given by 
\begin{align}\label{eta2}
\eta(r)=\begin{cases} r^2 &\mbox{if } |r|\leq (2\max\left(\|u_1\|_{L^\infty(\Omega_T)},\|u_2\|_{L^\infty(\Omega_T)})\right))^{1/2} \\ 
\bar\eta & \mbox{if } |r|>1 \end{cases}
\end{align}
and $\bar \eta $ is a smooth symmetric increasing function growing like $r^\alpha$ when $r\to\infty $ with $\alpha< s$.

Since $u_1$ and $u_2$ are bounded is easy to see that $\Phi$ achieves its global maximum at a point $(\bar x,\bar t,\bar y,\bar \tau)$ in $\bar\Omega_T^2$. Denote $\bar M=\Phi(\bar x,\bar t,\bar y,\bar \tau)$

We need to extract some information on the rate of convergence of the penalization. For this observe that the inequality
\[
\bar M=\Phi(\bar x,\bar t,\bar y,\bar \tau)\geq  u_1(x,t)-u_2(x,t),
\]
and we conclude, by taking the supremum, $M\leq \bar M$. From here we get 
\[
M\leq u_1(\bar x, \bar t)-u_2(\bar y, \bar \tau)-\eta\left(\frac{|\bar x-\bar y|}{\varepsilon_1^{1/2}}\right)-\frac{(\bar t-\bar \tau)^2}{\varepsilon_2}
\]
and since $M>0$ we get
\[
\eta\left(\frac{|\bar x-\bar y|}{\varepsilon_1^{1/2}}\right)+\frac{(\bar t-\bar \tau)^2}{\varepsilon_2}\leq 2\max\left(\|u_1\|_{L^\infty(\Omega_T)},\|u_2\|_{L^\infty(\Omega_T)})\right)
\]
which implies as in Theorem \ref{CPC}, $|\bar x-\bar y|, |\bar t-\bar \tau|\to 0$ as $\varepsilon_1,\varepsilon_2\to 0$. Now thanks to the continuity of $u_1$ and $u_2$ we deduce that $\bar M\to M$ as $\varepsilon_1,\varepsilon_2\to 0$. This can be seen from 
\begin{align*}
M&\leq u_1(\bar x, \bar t)-u_2(\bar y, \bar s)-\eta\left(\frac{|\bar x-\bar y|}{\varepsilon_1^{1/2}}\right)-\frac{(\bar t-\bar \tau)^2}{\varepsilon_2}\\
&\leq u_1(\bar x, \bar t)-u_2(\bar y, \bar \tau)\to M
\end{align*}
Finally, taking into consideration the special structure of $\eta$ and the fact that $u_1(\bar x, \bar t)-u_2(\bar y, \bar \tau)\to M$ we deduce the rates
\begin{align}\label{ratevar2}
|\bar t-\bar \tau|\to 0 \text{ as }\varepsilon_2\to0.
\end{align}
and
\begin{align}\label{ratevar3}
\frac{|x_\varepsilon-y_\varepsilon|^2}{\varepsilon}\to 0,\quad \text{as }\varepsilon\to 0^+.
\end{align}

Note also that $(\bar x,\bar t)$ and $(\bar y,\bar \tau)$ belong to $\Omega\times(0,T]$ for $\varepsilon_1, \varepsilon_2$ sufficiently small. If not, then there is a converging subsequence of $(\bar x,\bar t)$, $(\bar y,\bar \tau)$ so that $(\bar x, \bar t)\to (x,t)\in\Omega^c\times(0,T]\cup\R^N\times\{0\}$ and $\bar y, \bar \tau\to (x,t)\in\Omega^c\times(0,T]\cup\R^N\times\{0\}$. Since $\bar M\to M$ then the continuity of $u_1,u_2$ and the hypotheses of the theorem implies $M=u_1(x,t)-u_2(x,t)\leq 0$, which contradicts the fact $M>0$.

Let 

\[
\varphi_1(x,t)= u_2(\bar y, \bar \tau)+\eta\left(\frac{|x-\bar y|^2}{\varepsilon_1^{1/2}}\right)+\frac{(t-\bar \tau)^2}{\varepsilon_2}
\] 
and 
\[
\varphi_2(y,\tau)=u_1(\bar x, \bar t)-\eta\left(\frac{|\bar x-y|}{\varepsilon_1^{1/2}}\right)-\frac{(\bar t-\tau)^2}{\varepsilon_2}+\delta\rho(\bar x, y)
\]
Assume now that $0<\bar t,\bar \tau\leq T$, $\bar x, \bar y\in\Omega$ for all $\varepsilon_1,\varepsilon_2$  sufficiently small, then it is easy to see that $\varphi_1$ is a test function for $u_1$ at $(\bar x,\bar t)$ and analogously $\varphi_2$ is a test function for $u_2$ at $(\bar y, \bar \tau)$. Since $u_1$ and $u_2$ are viscosity sub and super solutions we get
\[
2\frac{(\bar t-\bar \tau)}{\varepsilon_2}+H\left(\bar x, \bar t,D_s\left(\eta\left(\frac{|\cdot-\bar y|^2}{\varepsilon_1^{1/2}}\right)\right)(\bar x)\right)< -\rho
\]
and
\[
2\frac{(\bar t-\bar \tau)}{\varepsilon_2}+H\left(\bar y, \bar \tau,D_s\left(\eta\left(\frac{|\bar x-\cdot y|^2}{\varepsilon_1^{1/2}}\right)\right)(\bar y)\right)\geq 0
\]
Recall from the proof of Theorem \ref{CPC} that 
\[
D_s\left(\eta\left(\frac{|\bar x-\cdot|^2}{\varepsilon_1^{1/2}}\right)\right)(\bar y)=D_s\left(\eta\left(\frac{|\cdot-\bar y|^2}{\varepsilon_1^{1/2}}\right)\right)(\bar x)
\]
Subtracting the previous inequality and utilizing hypothesis \eqref{PHJ1} we deduce
\begin{align*}
\tau&\leq H\left(\bar y, \bar \tau,D_s\left(\eta\left(\frac{|\bar x-\cdot y|^2}{\varepsilon_1^{1/2}}\right)\right)(\bar y)\right)-H\left(\bar x, \bar t,D_s\left(\eta\left(\frac{|\cdot-\bar y|^2}{\varepsilon_1^{1/2}}\right)\right)(\bar x)\right)\\
&\leq H\left(\bar y, \bar \tau,D_s\left(\eta\left(\frac{|\bar x-\cdot y|^2}{\varepsilon_1^{1/2}}\right)\right)(\bar y)\right)-H\left(\bar x, \bar t,D_s\left(\eta\left(\frac{|\bar x-\cdot y|^2}{\varepsilon_1^{1/2}}\right)\right)(\bar y))\right)\\
&\leq \omega_1\left(|\bar t-\bar \tau|+ |\bar x-\bar y|\left(1+D_s\left(\eta\left(\frac{|\bar x-\cdot y|^2}{\varepsilon_1^{1/2}}\right)\right)(\bar y)\right)\right).
\end{align*}
From here we arrive at a contradiction by noting that, as in the proof of Theorem \ref{CPC},
\[
|\bar x-\bar y|\left(1+D_s\left(\eta\left(\frac{|\bar x-\cdot y|^2}{\varepsilon_1^{1/2}}\right)\right)(\bar y)\right)\to 0\quad\text{as }\varepsilon_1\to 0.
\]
and that $|\bar t -\bar \tau|\to 0$ as $\varepsilon_2\to 0$.
\end{proof}

\begin{remark}
The proof of Theorem \ref{TeoP} in the case $\Omega=\R^N$ is more delicate. In the local case the lack of compactness is overcome by adding a cut-off function in the function $\Phi$. The main issue with this argument in the nonlocal case is that the associated test function $\varphi_1, \varphi_2$ will not have the same nonlocal gradient and therefore we can't compare the associated equations.
\end{remark}

\noindent {\bf Acknowledgements.} 

G. D. was partially supported by Fondecyt Grant No. 1190209.


\begin{thebibliography}{00}

\bibitem{Arnold}
Arnold, V. I. {\em Mathematical methods of classical mechanics.} Translated from the Russian by K. Vogtmann and A. Weinstein. Second edition. Graduate Texts in Mathematics, 60. Springer-Verlag, New York, 1989. xvi+508 pp

\bibitem{A}
Abdellaoui, B., Fern\'andez, A. J. {\em Nonlinear fractional Laplacian problems with nonlocal ``gradient terms".} To appear in Proc. Royal Soc. Edinburgh Sect. A.

\bibitem{Arisawa}
Arisawa, M. {\em A remark on the definitions of viscosity solutions for the integro-differential equations with L\'evy operators}  J. Math. Pures Appl. (9) 89 (2008), no. 6, 567-574.

\bibitem{BE}
Bakry, D.; \'Emery, M. {\em Diffusions hypercontractives.} In S\'eminaire de Probabilit \'es, XIX, 1983/84.Lecture Notes in Math. 1123 177-206. Springer, Berlin.

\bibitem{BGL}
Bakry, D.; Gentil, I.; Ledoux, M. {\em Analysis and geometry of Markov diffusion operators.} Grundlehren der Mathematischen Wissenschaften [Fundamental Principles of Mathematical Sciences], 348. Springer, Cham, 2014. xx+552 pp. 

\bibitem{BC}
Bardi, M., Capuzzo-Dolcetta, I. {\em Optimal control and viscosity solutions of Hamilton-Jacobi-Bellman equations.} Systems \& Control: Foundations \& Applications. Birkh\"auser Boston, Inc., Boston, MA, 1997. xviii+570 pp. 


\bibitem{B}
Barles, G. {\it Existence results for first order Hamilton Jacobi equations.} Ann. Inst. H. Poincaré Anal. Non Linéaire 1 (1984), no. 5, 325-340.


\bibitem{B2}
Barles, G. {\it Uniqueness for first-order Hamilton-Jacobi equations and Hopf formula.} J. Differential Equations 69 (1987), no. 3, 346-367.

\bibitem{B3}
Barles, G. {\it Solutions de viscosit\'e des \'equations de Hamilton-Jacobi.} Math\'ematiques \& Applications (Berlin), 17. Springer-Verlag, Paris, 1994

\bibitem{Barles}
Barles, G. {\em An introduction to the theory of viscosity solutions for first-order Hamilton-Jacobi equations and applications.}  Hamilton-Jacobi equations: approximations, numerical analysis and applications, 49-109, Lecture Notes in Math., 2074, Fond. CIME/CIME Found. Subser., Springer, Heidelberg, 2013. 

\bibitem{BCI}
Barles, G.; Chasseigne, E.; Imbert, C. {\em On the Dirichlet problem for second-order elliptic integro-differential equations.} Indiana Univ. Math. J. 57 (2008), no. 1, 213-246.

\bibitem{BI}
Barles, G.; Imbert, C. {\em Second-order elliptic integro-differential equations: viscosity solutions' theory revisited.} Ann. Inst. H. Poincar\'e Anal. Non Lin\'eaire 25 (2008), no. 3, 567-585.

\bibitem{BSoug}
Barles, G.; Souganidis, Panagiotis E. {\it On the large time behavior of solutions of Hamilton-Jacobi equations.} SIAM J. Math. Anal. 31 (2000), no. 4, 925-939.



\bibitem{BM}
Barrios, B., Medina, M. {\em Equivalence of weak and viscosity solutions in fractional non-homogeneous problems.} arXiv:2006.08384 

\bibitem{BDS}
Banerjee, A., D\'avila, G., Sire, Y. {\em Regularity for parabolic systems with critical growth in the gradient and applications.} arXiv:2005.04004




\bibitem{CaD}
Caffarelli, L. A., D\'avila, G. {\em Interior regularity for fractional systems}. Annales de l'Institut Henri Poincar\'e C, Analyse non lin\'eaire, Volume 36, Issue 1, 2019, Pages 165-180.

%
%


%

\bibitem{CJ}
Chasseigne, E.; Jakobsen, E. R. {\em On nonlocal quasilinear equations and their local limits.} J. Differential Equations 262 (2017), no. 6, 3759-3804.

\bibitem{CIs}
Crandall, M. G.; Ishii, H. {\em The maximum principle for semicontinuous functions.} Differential Integral Equations 3 (1990), no. 6, 1001-1014.


\bibitem{DQT}
D\'avila, G., Quaas, A., Topp, E.{\em Harnack Inequality and self-similar solutions for fully nonlinear fractional parabolic equations.} arXiv:1909.02624 

\bibitem{Davila-Topp}
D\'avila, G., Topp, E. {\em The Nonlocal Inverse Problem of Donsker and Varadhan.}  	arXiv:2011.13295

\bibitem{DLP}
Da Lio, F.; Pigati, A. {\em Free boundary minimal surfaces: a nonlocal approach}. To appear in Annali SNS (2009).  

\bibitem{DR}
Da Lio, F.; Rivi\`ere, T. {\it Three-term commutator estimates and the regularity of 1/2-harmonic maps into spheres.} Anal. PDE 4 (2011), no. 1, 149-190. 

\bibitem{DR2}
Da Lio, F.; Rivi\`ere, T. {\it Sub-criticality of non-local Schr\"odinger systems with antisymmetric potentials and applications to half-harmonic maps.} Adv. Math. 227 (2011), no. 3, 1300-1348.

\bibitem{Dyda}
Dyda, B. {\em Fractional calculus for power functions and eigenvalues of the fractional Laplacian.}  Fract. Calc. Appl. Anal. 15 (2012), no. 4, 536-555.

%
%
%
%

%
%

\bibitem{Ishii}
Ishii, H. {\it Perron's method for Hamilton-Jacobi equations.} Duke Math. J. 55 (1987), no. 2, 369-384.

\bibitem{Is1}
Ishii, H. {\em On the equivalence of two notions of weak solutions, viscosity solutions and distribution solutions.} Funkcial. Ekvac. 38 (1995), no. 1, 101-120. 

\bibitem{JK}
Jakobsen, E. R.; Karlsen, K. H.{\em A``maximum principle for semicontinuous functions" applicable to integro-partial differential equations.} NoDEA Nonlinear Differential Equations Appl. 13 (2006), no. 2, 137-165.

\bibitem{MR}
Ma, Z. M., R\"ockner, M. {\em Introduction to the theory of (nonsymmetric) Dirichlet forms.} Universitext. Springer-Verlag, Berlin, 1992. vi+209 pp. 

%
%
\bibitem{Millot-Sire}
Millot, V., Sire, Y. {\em On a fractional Ginzburg-Landau equation and 1/2-harmonic maps into spheres.} Arch. Ration. Mech. Anal. 215 (2015), no. 1, 125-210.

\bibitem{NJ}
Namah, Gawtum; Roquejoffre, Jean-Michel. {\it Remarks on the long time behaviour of the solutions of Hamilton-Jacobi equations.} Comm. Partial Differential Equations 24 (1999), no. 5-6, 883-893.

%
%
%
%

\bibitem{RS}
Ros-Oton, X., Serra, J.{\em The Dirichlet problem for the fractional Laplacian: Regularity up to the boundary.} Journal de Mathématiques Pures et Appliqu\'ees, Vol. 101, Issue 3, 2014, pp. 275-302,

\bibitem{RS2}
Ros-Oton, X., Serra, J. {\em The Pohozaev Identity for the Fractional Laplacian}.  Arch Rational Mech Anal 213, 587-628 (2014). 

\bibitem{SWZ}
Spener, A.; Weber, F.; Zacher, R. {\em The fractional Laplacian has infinite dimension.} Comm. Partial Differential Equations 45 (2020), no. 1, 57-75. 


\end{thebibliography}
\end{document}